\documentclass[a4paper]{article}
\usepackage[applemac]{inputenc}
\usepackage{amsmath}
\usepackage{amsthm}
\usepackage{amssymb} 

\theoremstyle{plain}
\newtheorem{thm}{Theorem}
\newtheorem{cor}[thm]{Corollary}
\newtheorem{lem}[thm]{Lemma}
\newtheorem{clm}{Claim}

\theoremstyle{remark}
\newtheorem*{rem*}{Remark}

\theoremstyle{definition}

\newtheorem*{exm*}{Example}

\DeclareMathOperator{\coloneq}{\mathrel{\mathop:}=}
\DeclareMathOperator{\tw}{tw}
\DeclareMathOperator{\ord}{ord}

\newcommand{\D}{\ensuremath{\mathcal{D}}}

\newcommand{\W}{\ensuremath{\mathcal{W}}}

\renewcommand{\P}{\ensuremath{\mathcal{P}}}

\newcommand{\boundary}{\partial}

\begin{document}

\title{Linear Connectivity Forces Large Complete Bipartite Minors: the Patch for the Large Tree-Width Case}
\author{Jan-Oliver Fröhlich \and Theodor Müller}
\date{University of Hamburg\\June 2009}
\maketitle


\abstract{The recent paper `Linear Connectivity Forces Large Complete Bipartite Minors' by Böhme et al.\ relies on a structure theorem for graphs with no $H$-minor.
The sketch provided of how to deduce this theorem from the work of Robertson and Seymour appears to be incomplete.
To fill this gap, we modify the main proof of that paper to work with a mere restatement of Robertson and Seymour's original results instead.}

\section{Introduction}
Robertson and Seymour proved several variants of a structure theorem for graphs with no $H$-minor.
The version in~\cite{gm17} roughly says that any pair of such a graph $G$ and a tangle of sufficiently high order in $G$ has a near-embedding in some surface in which $H$ cannot be embedded.
This structure theorem guarantees that no bag of a vortex decomposition of the near-embedding contains a large side of any separation in the given tangle.
In their recent paper~\cite{kak}, Böhme et al.\ propose a strengthening of the above structure theorem for graphs of large tree-width, and make this the cornerstone of the proof of their main result.

Recall that a graph $G$ of large tree-width must  contain a large wall as a minor (cf.~\cite{gm5}).
The ``location'' of this wall can be encoded into a large-order tangle in $G$.
Applying the Robertson-Seymour structure theorem to $G$ and this tangle, the authors of~\cite{kak} wish to conclude not only that no vortex bag contains a large side of a tangle separation as the theorem states, but even that no vortex contains large parts of the wall, and that therefore there must be a large subwall essentially embedded in the surface.
However, these last claims are not backed up by a rigorous proof.

Our aim is to fill the gap between the work of Robertson and Seymour and the proof of the main result of Böhme et al.\ in~\cite{kak}.
We assume the reader to be familiar with the latter paper.
Any references to its sections, theorems, etc.\ will be \emph{emphasized} to distinguish them from references within our paper.

Instead of proving the structure theorem (\emph{Theorem~4.2}) of~\cite{kak} in general, we modify the proof of the main result (\emph{Theorem~1.1}) to work with the original structure theorem of Robertson and Seymour itself, restated for readability in a terminology similar to~\cite{kak}.
Our patch will thus be a short argument which fits smoothly into the proof of \emph{Theorem~1.1}.

We remark that, in the meantime, \emph{Theorem~4.2} has indeed been found to be true.
But it is non-trivial to prove even on the basis of the work of Robertson and Seymour.
A proof of a strengthened general version will be published separately~\cite{newstructure}.

Here is an overview of the layout of our note.
In Section~2, we present the Robertson-Seymour structure theorem we need, in terminology close to~\cite{kak}.
Section~3 contains the tools necessary to exploit that structure theorem in the setting of \emph{Theorem 1.1}.
In Section~4, we redo the proof of \emph{Theorem~1.1} up to point where a wide vortex is found (\emph{Claim~5.3}).
After that, no further modifications are required.

\section{The Structure Theorem}
There are several variants of the structure theorem for graphs with no $H$-minor in the `Graph Minors' series.
We shall use the one in~\cite{gm17}.
To state it requires a fair amount of preparation.
Note that Böhme et al.\ based their structure theorem on~\cite{gm17} as well. 

A pair $(G, \Omega)=:V$ of a graph $G$ and a linearly ordered subset $\Omega\subseteq V(G)$ is called a \emph{vortex}. 
The vertices $\Omega$ are the \emph{society vertices} of the vortex and their number $|\Omega|$ is its \emph{length}. For convenience we also write $G$ for the vortex. 
The vertices in $V(G)\setminus \Omega$ are the \emph{inner} vertices of~$V$. 
A vortex without inner vertices is \emph{trivial}. 
We enumerate a society $\Omega=\{w_1,\ldots, w_n\}$ always according to its order, i.e.\ $w_1<w_2<\ldots <w_n$.

A path-decomposition $\D = (X_1,\ldots,X_m)$ of $G$ is a \textit{decomposition of~$V$} if $m=n$ and $w_i\in X_i$ for all~$i$. The $X_i$ are the \emph{bags} or \emph{parts} of the decomposition.

Given a decomposition $\D = (X_1,\ldots,X_m)$ of a vortex $V$ as above. For all $1 \leq i< n$, we write ${Z_i:=(X_{i}\cap X_{i+1})\setminus \Omega}$ and call $\D$ \emph{linked} if 
\begin{itemize}
	\item all these $Z_i$ have the same size, $q$ say;
	\item there are $q$ disjoint $Z_{i-1}$--$Z_{i}$ paths in~$G[X_i]-\Omega$, for all $1<i<n$;
	\item $X_i \cap \Omega = \{w_{i-1}, w_{i}\}$ for $1\leq i\leq  n$, where $w_0:= w_1$.
\end{itemize}
Note that $X_i\cap X_{i+1} = Z_i \cup \{w_i\}$, for all $1\le i < n$

The union of the $Z_{i-1}$--$Z_{i}$ paths in a linked decomposition of $V$ is a disjoint union of $X_1$--$X_n$ paths in~$G$; we call the set of these paths a \emph{linkage} of~$V$ with respect to $(X_1,\ldots,X_m)$.

The union of a path $P$ with some mutually disjoint paths, having precisely their first vertex on $P$, is a \emph{comb}; 
the last vertices of those paths are the \emph{teeth} of this comb.

A graph $G$ is said to be $\alpha$-\emph{near-embeddable} in some surface $\Sigma$ if there is a subset $A\subseteq V(G)$ of at most $\alpha$ vertices, the \emph{apex} set, such that $G-A$ is an edge-disjoint union of subgraphs $G_0,\ldots, G_n$ of $G$ with integers $0\leq\alpha'\leq \alpha\leq n$ such that:
\begin{enumerate}
	\item The pairs $(G_i, \Omega_i)$ where $\Omega_i\coloneq V(G_i\cap G_0)$ with some linear order are non-trivial vortices and different vortices overlap only in $G_0$: we have $G_i\cap G_j\subseteq G_0$ for $i\neq j$.
	\item $G_1, \ldots, G_{\alpha'}$ are disjoint and have linked decompositions of adhesion at most $\alpha$. These are the \emph{large} vortices and will be denoted by $\mathcal{V}$. For each of these vortices we fix a linked decomposition together with a linkage and whenever we refer to the decomposition or the linkage of a given large vortex, we shall mean these fixed ones.
	\item The remaining vortices $G_{\alpha'+1},\ldots, G_n$ have length at most $3$. They are called \emph{small} vortices\footnote{These small vortices $W\in\W$ represent subgraphs in \cite{kak} that are split off of $G$ along separators of order at most $3$ in `elementary reductions' while $\Omega(W)$ are the vertices `involved'  in this reduction.} and we denote them by $\mathcal{W}$.
	\item For each large vortex $G_j$ there is a comb $K$ in $G_j\cup\bigcup\W$  which is disjoint to the linkage of $G_j$ and whose teeth are the society vertices $\Omega_j$ in the right order.
	\item There are closed discs in~$\Sigma$ with disjoint interiors $D_1,\ldots, D_n$ and an embedding ${\sigma: G_0 \hookrightarrow \Sigma-\bigcup_{i=1}^n D_i}$ such that $\sigma(G_0)\cap\boundary D_i = \sigma(\Omega_i)$ for all~$i$ and the generic linear ordering of $\Omega_i$ is compatible with the natural cyclic ordering of its image (i.e., coincides with the linear ordering of $\sigma(\Omega_i)$ induced by $[0,1)$ when $\boundary D_i$ is viewed as a suitable homeomorphic copy of $[0,1]/\{0,1\}$). 
\end{enumerate}
The tuple $(\sigma, A, G_0, \mathcal{V}, \mathcal{W})$ is an $\alpha$-\emph{near-embedding} or just \emph{near-embedding} of $G$ in $\Sigma$. A near-embedding with apex set $A$ is said to \emph{respect} a tangle $\mathcal{T}$ if no large side of any separation in $\mathcal{T}\setminus A$ is contained in a vortex $W\in\mathcal{W}$ or in a bag of the decomposition of a vortex $V\in\mathcal{V}$. 

We will combine results from~\cite{gm17}, namely the structure theorem (13.4) and two lemmas~(9.1) and (9.8) to show the following 
\begin{thm}[structure theorem]\label{structure}
	For every graph $H$ there exist non-negative integers $\theta$ and $\alpha$ such that the following holds: let $G$ be a graph not containing $H$ as a minor and $\mathcal{T}$ any tangle of order at least $\theta$ in $G$. Then $G$ has an $\alpha$-near-embedding into a surface $\Sigma$ in which $H$ cannot be embedded and this near-embedding respects $\mathcal{T}$.
\end{thm}
In the remainder of this section, we will show how to deduce Theorem~\ref{structure} from the results mentioned above. For this purpose the readers should be familiar with the concepts of \cite{gm17} such that they understand the statement (13.4). In particular, we will explicitly use the functions $\alpha$ and $\beta$ which are part of the definition of a \emph{portrayal}. 


The application of (13.4) yields a portrayal which translates directly  in our concept of a near-embedding---with one exception: 
Some technical arguments are necessary to transform the subgraphs bordering a cuff in the surface, which have a `circular' structure, into large, linked vortices, which have by their decompositions a `linear' structure. Clearly, there is an integer $r$ bounding all the constants given by  (13.4): The `apex' set $Z_0\subseteq V(G)$ has size at most $r$, and the portrayal of $G-Z_0$ of warp at most $r$ lives in a surface having most $r$ cuffs. For our intended conversion we need to delete additional $r^2$ vertices from $G$. Thus, with $\alpha:=r+r^2$, we will eventually obtain an $\alpha$-near-embedding of $G$ with a larger apex set of size at most $\alpha$.

Let us consider all border nodes $(w_0,\ldots,w_n)$ of a given cuff, linearly ordered in a way compatible with their cyclic ordering on the cuff. Let $c_i$ be the border cell with ends $w_{i-1}$ and $w_i$ for $0\leq i \leq n$ and $w_{-1}:=w_n$. Further, let $X_i$ denote the vertex set of $\alpha(c_i)$ and $R$ be the graph $\bigcup_i \alpha(c_i)$. We will convert this graph $R$ into a large vortex for our near-embedding. 

Lemmas (9.8) and (9.1) tell us that $\beta(w_0),\ldots,\beta(w_n)$ all have the same size, $q\leq \alpha$ say, and for $0\leq i\leq n$ there are disjoint paths $P_0^i,\ldots,P_q^i$ connecting $\beta(w_{i-1}+)$ to $\beta(w_i+)$ with the following properties:
\begin{itemize}
\item All the paths $P_1^i,\ldots,P_q^i$ are contained in $\alpha(c_i)$.
\item $P_0^i$ connects $w_{i-1}$ and $w_i$ and avoids all other society vertices.
\item $P_0^i$ is either contained in $\alpha(c_i)$ or in $\alpha(c')$ for some internal cell $c'$ with $w_{i-1},w_i\in \tilde c'$.
\end{itemize}
Let us assume that the paths are enumerated such that for $1\leq k \leq q$ and $1\leq i \leq n$ the last vertex of $P_k^{i}$ is the initial vertex of  $P_k^{i+1}$.

Let $P_0$ denote the union of the paths $P_0^2,\ldots,P_0^{n}$. This graph contains a comb which we will need for our construction to satisfy property~4 of near-embeddable. 

For each $k=1,\ldots,q$, let $P_k$ be the union of the paths $P_k^0,\ldots,P_k^n$ and let $\P$ denote the set of these $P_k$.  
We may assume that $|P_k|=1$ if and only if $k>q'$ for some $q'\leq q$. 
For $1\leq k\leq q'$ the graphs in $\P$ are paths or cycles 
and we regard them as oriented according to the order in which they traverse the $\alpha(c_i)$. Now, each vertex $v\in\beta(w_0)$ not lying on some trivial path in $\P$ is either contained in one cycle or in two paths and thus, has a unique successor $v_s$ and a unique predecessor $v_p$ in $\P$. 
Now $Z:=\beta(w_0+)$ is a set of at most $\alpha$ vertices; 
deleting $Z$ from $H$ yields a graph $H'$ and with (P6) and (P7) from the definition of a portrayal it follows that $(X_1\setminus Z,\ldots,X_{n-1}\setminus Z,(X_n\cup X_0)\setminus Z)$ is a decomposition of the vortex $(H',\{w_1,\ldots,w_n\})$ of adhesion at most $q'$. 

With $\P':=\{P_i-Z \mid 1\leq i \leq q'\}$ we have a system of disjoint paths connecting the adhesion sets $Z_i$ of our decomposition. 
However, by the deletion of $Z$, some of the adhesion sets  $Z_i$ might be strictly smaller than $q'$ as the corresponding sets $\beta(w_i)$ also could have contained a vertex $v\in Z$. We can solve this problem by adding either $v_s$ or $v_p$ to $X_i$ and achieve that $\P'$ is a linkage of our vortex. It is easy to see with (9.1) that $P_0-Z$ contains a comb with teeth $\{w_1,\ldots,w_n\}$.

We have deleted up to $r$ vertices for this construction which is necessary for up to $r$ cuffs. Thus, adding all these  up to $r^2$ vertices to the apex set, Theorem~\ref{structure} holds for $\alpha=r^2+r$.


\section{Preparations}
To take full advantage of Theorem~\ref{structure}, we need a suitable tangle $\mathcal{T}$ in the graph $G$ as input.
Suitable means in our case that $\mathcal{T}$ has high order and encodes the location of a large grid minor in $G$.
The interplay of tangles, grids and grid minors is the subject of this section.

For a positive integer $r$ let $W_r$ be the grid on $r^2$ vertices.
More precisely, $W_r$ has vertices $V(W_r) \coloneq \{(i,j): 1\leq i,j\leq r\}$ and there is an edge between two vertices $(i,j)$ and $(i',j')$ of $W_r$ if and only if $|i-i'| + |j-j'| = 1$.
We call $\{1,\ldots, r\}\times \{j\}$ the $j$th \emph{column} of $W_r$ for $1\leq j\leq r$.
Similarly, $\{i\}\times \{1,\ldots r\}$ is the $i$th \emph{row}.
Uniting any row with any column gives a \emph{cross}.
The graph $W_r$ is the $r$-\emph{grid}.

Our first statement is that, for any positive integer $r$, the $r$-grid gives rise to a tangle of order $r$ in a canonical way.
Consider the set $\mathcal{T}$ of all separations $(A,B)$ of order less that $r$ in $W_r$ such that $B$ contains a cross.
According to Robertson and Seymour~\cite[(7.3)]{gm10}, this is indeed a tangle of order $r$ in $W_r$.
We refer to this tangle as the \emph{natural} tangle of $W_r$.

Intuitively, a small side in any separation of a natural tangle cannot contain many vertices without causing the separator to be large.
We formalize this fact in
\begin{lem}\label{gridcutW}
	Each separation $(A,B)$ in the natural tangle of some grid satisfies $|A|\leq s^2$ where $s\coloneq \ord (A,B)$.
\end{lem}

\begin{proof}
	Let $(A,B)$ be any separation in the natural tangle $\mathcal{T}$ of $W_r$.
	Let $s$ be the order of $(A,B)$.
	Denote by $I$ the set of all numbers $i$ such that $A$ has a vertex in the $i$th row.
	Similarly, let $J$ be the set of all $j$ such that $A$ has a vertex in the $j$th column.
	$B$ is large and therefore contains a cross.
	Hence the rows with index in $I$ contain\footnote{In a slight abuse of terminology we speak of a subgraph being contained in a vertex set.} $|I|$ disjoint $A$--$B$ paths.
	And the columns with index in $J$ contain $|J|$ disjoint $A$--$B$ paths.
	The separator $A\cap B$ must clearly have a vertex on each of those paths.
	This means $|I|, |J|\leq s$. But $A\subseteq I\times J$ and thus $|A|\leq |I| |J| \leq s^2$.
\end{proof}

Robertson and Seymour also provide a general way of extending a tangle $\mathcal{T}'$ in any minor $H$ of $G$ to a tangle $\mathcal{T}$ in $G$.
Although we present their construction in general, we shall use it only in the case where $H$ is a grid and $\mathcal{T}'$ its natural tangle.
The following two lemmas are based on~\cite[(6.1)]{gm10}.
They are both straightforward to verify so we spare the proofs.
Let $\mathcal{T}'$ be a tangle in some graph $H$ and $G$ any graph containing $H$ as a minor, witnessed by branch sets $V_h$ with $h\in V(H)$.
Then separations of $G$ induce separations of $H$:
\begin{lem}\label{inducedseparation}
	For any separation $(A,B)$ in $G$ the pair
	\[(A',B')\coloneq \left(\{h\in V(H): V_h\cap A\neq\emptyset\}, \{h\in V(H): V_h\cap B\neq\emptyset\}\right)\]
	is a separation of at most the same order as $(A,B)$ in $H$.
\end{lem}
The separation $(A',B')$ is said to be \emph{induced} in $H$ by $(A,B)$.
This enables us to (uniquely) extend $\mathcal{T}'$ to a tangle $\mathcal{T}$ of the same order in $G$:
\begin{lem}\label{extendedtangle}
	Let $\mathcal{T}$ be the set of all separations $(A,B)$ in $G$ of order less than $\ord\mathcal{T}'$ such that the induced separation $(A', B')$ lies in $\mathcal{T}'$.
	Then $\mathcal{T}$ forms a tangle of order $\ord\mathcal{T}'$ in $G$.
\end{lem}
We call the tangle $\mathcal{T}$ the \emph{extension} of $\mathcal{T}'$ to $G$.
Now we are able to extend the natural tangle $\mathcal{T}'$ of $W_r\preccurlyeq G$ to a tangle $\mathcal{T}$ of order $r$ in $G$.
By the definition of the extension, Lemma~\ref{gridcutW} obviously carries over as

\begin{cor}\label{gridcutG}
	Let $G$ be any graph containing some grid $W$ as a minor.
	Then the small side of any separation in the extension of the natural tangle of $W$ intersects at most $s^2$ branch sets of $W$, where $s$ is the order of the separation.
\end{cor}

\section{Finding a Wide Vortex}
In this section we provide an alternative opening for the proof of the main theorem (\emph{Theorem~1.1}) in~\cite{kak}.
It avoids the use of \emph{Theorem~4.2} and applies Theorem~\ref{structure} instead.
As Böhme et al.\ point out, their structure theorem serves two purposes in the original proof:
Making the embedded graph large and providing the ``path''~$P_0$.
Both is necessary only to ensure the existence of an $n_2$-wide vortex (\emph{Claim~5.3}).
We shall redo the proof in detail up to that wide vortex.

\emph{Theorem~1.1} clearly follows from the bounded tree-width theorem (\emph{Theorem~3.1}) and
\begin{thm}[large tree-width]\label{main}
	For any positive integers $a,s$ and $k$, there exists a constant $w=w(a,s,k)$ such that every graph $G$ with
	\[\kappa(G)\geq 3a+2,\qquad\delta(G)\geq \frac{31}{2}(a+1)-3,\qquad\text{and}\qquad \tw(G) > w\]
	contains $s$ disjoint $K_{a,k}$-minors or a subdivision of $K_{a,sk}$.
\end{thm}

For given values of $a$, $s$ and $k$ set $H\coloneq sK_{a,k}$. Theorem~\ref{structure} yields constants $\alpha$ and $\theta$ such that for any graph $G$ with no $H$-minor and a tangle $\mathcal{T}$ of order at least $\theta$ there is an $\alpha$-near-embedding of $G$ into some surface $\Sigma$ into which $H$ cannot be embedded and this near-embedding respects~$\mathcal{T}$.
Without loss of generality we may assume $\alpha > 1$.
Let the constants $n_2$, $n_3$, $n_4$ and $n_5$ be chosen as in~\cite{kak}.
Set $g\coloneq (sk-1)\binom{\alpha}{a}$ and let $n_1$ be sufficiently large to make
\[\frac{1}{7}n_1 -2(a+1)g - ask \geq 3g + (n_2-1)\alpha\]
true.
Let the integer $r$ be large enough that
\[r\geq\theta, \qquad r  > 3\alpha \qquad\text{and}\qquad r^2 > (n_1+g) (3\alpha)^2 + n_1.\]
Robertson and Seymour showed~\cite{gm5} that the tree-width of graphs not having the grid $W_r$ as a minor is bounded as a function of $r$.
So we can pick $w=w(r)$ such that any graph of tree-width larger than $w$ has the $r$-grid as a minor.
We claim that Theorem~\ref{main} holds with this choice of~$w$.
For suppose not. 
Then there is a counterexample $G$, i.e.\ a graph with connectivity at least $3a+2$, minimum degree at least $\frac{31}{2}(a+1)-3$ and tree-width larger than $w$ containing neither $s$ disjoint $K_{a,k}$-minors nor a subdivision of $K_{a,sk}$.

By the choice of $w$ the graph $G$ has the $r$-grid as a minor.
We fix the branch sets of such a minor for the remainder of the proof.
Let $\mathcal{T}$ be the extension of its natural tangle to $G$.
Since $G$ does not contain $sK_{a,k}$ as a minor and $\ord\mathcal{T} = r\geq\theta$ we may apply Theorem~\ref{structure} to $G$ and $\mathcal{T}$.
Let all the notation be just as in the definition of near-embedding.

The following two claims are the analogues to \emph{Claim~5.1} and \emph{Claim~5.2}.
Keeping in mind that our small vortices play the same role as elementary reductions in \cite{kak}, the original proofs work in our setting as well.
\begin{clm}\label{apexnb}
	At most $g$ vertices of $G$ have $a$ or more neighbors in $A$.
\end{clm}

\begin{clm}\label{fewsmallvortices}
	There are at most $g$ small vortices.
\end{clm}

We now want to find a large vortex whose society vertices send only few edges into $G_0$.
The first step towards this ``wide'' vortex is to show that the graph $G_0$ is large by applying Corollary~\ref{gridcutG} to the large grid minor in $G$.
Remember that, in the original proof, \emph{Theorem~4.2} ensured the existence of a large wall in the embedded graph, making the next claim trivial in that setting.

\begin{clm}\label{largeG0}
	$|G_0|> n_1$.
\end{clm}

\begin{proof}[Proof of Claim~\ref{largeG0}]
	Suppose not.
	Since large vortices are disjoint, there are at most $n_1$ society vertices of large vortices in $G_0$ and hence at most $n_1$ bags of large vortices in~$G$.
	Corollary~\ref{gridcutG} is our tool to count how many branch sets of $W_r$ are spread out over the vortices and the apex set.
	As it turns out, even together with the branch sets in $G_0$, this number is less than $r^2$, giving a contradiction:
	
	Any small vortex is separated from the rest of $G-A$ by its at most~$3$ society vertices.
	Clearly, $3 < 2\alpha < r - \alpha \leq \ord (\mathcal{T}\setminus A)$.
	So the separation specified above is in $\mathcal{T}\setminus A$, and, as the near-embedding respects the tangle $\mathcal{T}$, the vertex set of the small vortex is its small side.
	By the defintion of $\mathcal{T}\setminus A$ this means that adding the apex set on both sides of the separation gives a separation of $G$ which lies in $\mathcal{T}$.
	Then Corollary~\ref{gridcutG} ensures that at most $(3 + \alpha)^2 < (3\alpha)^2$ branch sets of $W_r$ have a vertex in this small vortex or the apex set.
	
	Similarly, any bag of the decomposition of a large vortex is separated by its at most two adhesion sets from the rest of $G-A$.
	Since each adhesion set has size at most $\alpha$, the same argument as above implies that at most $(3\alpha)^2$ branch sets of $W_r$ have a vertex in this bag or the apex set.
	Then by Claim~\ref{fewsmallvortices} and the choice of $r$, $G$ contains at most
	\[(n_1 + g) (3\alpha)^2 + n_1 < r^2\]
	branch sets of $W_r$, contradiction.
\end{proof}

We could end the patch here and let the proof of \emph{Claim~5.3} from~\cite{kak} take over.
But that proof is inaccurate---it does not take into account the non-essential society vertices.
In fact, the Euler formula argument from~\cite{kak} fails for the given definition of an essential vertex.

Here are the definitions we are going to use:
a society vertex $v\in \bigcup_{i=1}^n \Omega_i$ is called \emph{essential} if it has less than $7$ neighbors in $G_0$.
A large vortex $G_i$ is  $m$-\emph{wide} if its society contains at least $m$ essential vertices.
Unlike its counterpart in~\cite{kak} the latter definition does not require the existence of a comb with its teeth being essential society vertices in the right order.
This is because Theorem~\ref{structure} yields such a comb right away.
In particular, we need not apply Erd\H{o}s-Szekeres to construct it.

\begin{clm}\label{widevortex}
	There is an $n_2$-wide vortex.
\end{clm}

\begin{proof}[Proof of Claim~\ref{widevortex}]
$G_0$ is embedded in $\Sigma$ and hence satisfies the general Euler formula (cf.~\cite[B.2]{diestel})
\[|G_0| - \|G_0\| + l = 2 - \varepsilon\]
where $l$ is the number of faces of the embedded graph and $\varepsilon$ is the Euler genus of $\Sigma$.
To eliminate $l$ and $\varepsilon$ from the equation use the inequalities $3l\leq 2\|G_0\|$ and $\varepsilon < ask$.
The former holds as each edge bounds at most two faces while each face is bounded by at least three edges.
The latter is true because $H$ cannot be embedded in the surface $\Sigma$ and $\|H\| = ask$.
Hence
\[\tag{1}|G_0| + ask > |G_0| + \varepsilon = 2 + \|G_0\| - l > \frac{1}{3}\|G_0\|.\]
Now partition the vertices of $G_0$ into three classes:
essential society vertices $X$, non-essential society vertices $Y$ and non-society vertices $Z$.
Denote the sizes of $X$, $Y$ and $Z$ by $x$, $y$ and $z$, respectively.
Then clearly
\[|G_0| = x + y + z.\]
The vertices in $Y$ each have at least $7$ neighbours in $G_0$.
By Claim~\ref{apexnb} at least $z-g$ vertices of $Z$ each send at least $\delta(G) - (a-1) \geq 14(a+1)$ edges into $G_0$.
In total we get
\[\tag{2}\frac{2}{7}\|G_0\|\geq y + 2(a+1)(z-g) \geq |G_0| - x - 2(a+1)g.\]
Together $(1)$ and $(2)$ imply
\[\frac{6}{7}\left(|G_0| + ask\right) > |G_0| - x - 2(a+1)g.\]
Since $G_0$ has at least $n_1$ vertices, $x$ is large enough to force one vortex to be $n_2$-wide by the choice of $n_1$:
\[x > \frac{1}{7}|G_0| -2(a+1)g - ask \geq 3g + (n_2-1)\alpha.\]
\end{proof}

Finally, by Claim~\ref{widevortex} one of the vortices, $G_j$ say, is $n_2$-wide.
By Theorem~\ref{structure} there is a comb $P_0$ which is disjoint to the linkage of $G_j$ and has teeth $\Omega_j$.
The rest works exactly as in the original proof.
In particular, our adjusted definition of an essential vertex causes no problems later on.


\begin{thebibliography}{99}
	\bibitem{kak}
		\textsc{T.~Böhme, K.~Kawarabayashi, J.~Maharry} and \textsc{B.~Mohar}:
		Linear Connectivity Forces Large Complete Bipartite Minors,
		\textit{J.~Combin.\ Theory~B} \textbf{99} (2009), 557--582.
	\bibitem{diestel}
		\textsc{R.~Diestel}:
		\textit{Graph Theory} (3rd edition), GTM 173, Springer-Verlag, Heidelberg (2005).
	\bibitem{newstructure}
		\textsc{R.~Diestel, K.~Kawarabayashi, T.~Müller} and \textsc{P.~Wollan}:
		On the excluded minor theorem in graphs of large treewidth,
		\textit{in preparation}.
	\bibitem{gm5}
		\textsc{N.~Robertson} and \textsc{P.~D.~Seymour}:
		Graph Minors. V. Excluding a Planar Graph,
		\textit{J.~Combin.\ Theory~B} \textbf{41} (1986), 92--114.
	\bibitem{gm10}
		\textsc{N.~Robertson} and \textsc{P.~D.~Seymour}:
		Graph Minors. X. Obstructions to Tree-Decomposition,
		\textit{J.~Combin.\ Theory~B} \textbf{52} (1991), 153--190.
	\bibitem{gm17}
		\textsc{N.~Robertson} and \textsc{P.~D.~Seymour}:
		Graph Minors. XVII. Taming a Vortex,
		\textit{J.~Combin.\ Theory~B} \textbf{77} (1999), 162--210.
\end{thebibliography}
\end{document}